\documentclass{article}
\usepackage{amssymb}
\usepackage{amsmath}
\usepackage{amsthm}
\usepackage{verbatim}
\usepackage{color}
\usepackage{url}
\usepackage{fullpage}

\newtheorem{thm}{Theorem}

\newtheorem{lem}{Lemma}
\newtheorem{cor}{Corollary}

\newcommand{\sabs}[1]{\left|#1\right|}
\newcommand{\sparen}[1]{\left(#1\right)}

\numberwithin{equation}{section}

\title{Isospectral Periodic Torii in Dimension 2}

\author{Alden Waters \\ CNRS Ecole Normale Superieure\\ 45 Rue d'Ulm\\ Paris, France 75005}
\begin{document}
\maketitle

\begin{abstract}
We consider two dimensional real-valued analytic potentials for the Schr\"odinger equation which are periodic over a lattice $\mathbb{L}$. Under certain assumptions on the form of the potential and the lattice $\mathbb{L}$, we can show there is a large class of analytic potentials which are Floquet rigid and dense in the set of $C^{\infty}(\mathbb{R}^2/\mathbb{L})$ potentials. The result extends the work of Eskin et. al, in "On isospectral periodic potentials in $\mathbb{R}^n$, II."\\
inverse spectral theory for Schr\"odinger operators. 35J10, 35P05, 65M32
\end{abstract}

\section{Introduction}
The subject of multi-dimensional inverse spectral theory has seen a small amount of growth in the past few decades after the work of Eskin et. al, in \cite{ert1} and \cite{ert2} in the context of Floquet rigidity. The reason for this is that it is difficult to calculate exactly the structure of spectral invariants for multi-dimensional periodic Schr\"odinger operators. The authors of \cite{ert1} and \cite{ert2} essentially are only able to consider perturbations of the zero potential in their work. The goal of this paper is to show that a larger class of analytic periodic potentials can be considered by use of the abelian functionals. Its and Mateev \cite{its} have shown that the abelian functionals categorize all finite gap potentials.  

The focus of this paper is the class of Schr\"odinger operators 
\[
P: u(x)\mapsto (-\Delta +q(x))u(x),
\]
where 
\[
\Delta=\sum\limits_{j=1}^2\frac{\partial^2}{\partial x_j^2}, 
\] 
and
 \[
 q(x):\mathbb{R}^2\rightarrow \mathbb{R}
 \]
 is a real-valued periodic potential over a lattice, $\mathbb{L}\subset\mathbb{R}^2$. In other words we have  
 \[
 q(x+d)=q(x) \qquad \forall d\in \mathbb{L}.
 \]
 We will study the question of spectral rigidity for the operator $P$ and derive results which could extend to $\mathbb{R}^n$ for $n\geq 3$. We consider the set of $\lambda$ in $\mathbb{R}$ for which the self-adjoint eigenvalue problem 
\begin{align}\label{selfadj}
&Pu(x)=\lambda u(x) \qquad u(x+d)=\exp\sparen{2\pi i k\cdot d} u(x)
\end{align}
has a solution for $k$ in $\mathbb{R}^2$ and $d$ in $\mathbb{L}$. When there is a nonzero solution to (\ref{selfadj}) we say that $\lambda$ is in Spec$_k(-\Delta +q)$. We refer to 
\[
\bigcup_{k\in \mathbb{R}}\mathrm{Spec}_k(-\Delta+q)
\]
as the Floquet spectrum.  However, when $k=0$, we simply say 'spectrum' which we denote by Spec$(-\Delta+q)$. Two potentials $q$ and $\tilde q$ are Floquet isospectral if 
\[
\mathrm{Spec}_k(-\Delta +q)=\mathrm{Spec}_k(-\Delta+\tilde{q})\qquad \forall k\in \mathbb{R}^2
\]
and isospectral if Spec$(-\Delta +q)=$Spec$(-\Delta +\tilde{q})$. Following the convention in \cite{ert2}, we consider a potential to be Floquet (spectrally) rigid if there are only a finite number of potentials modulo translations which are Floquet isospectral (resp. isospectral) to it.  

In  \cite{ert1}, Eskin et al. showed that under the assumptions
\begin{enumerate}
\item q is real analytic 
\item $\mathbb{L}$ has the property $|d|=|d'|\Rightarrow$ $d=\pm d'$ for all $d,d'$ in $\mathbb{L}$
\end{enumerate}
then Spec$(-\Delta +q)$ determines Spec$_k(-\Delta +q)$ for all $k$ in $\mathbb{R}^n$. 

 It is important to note that we are considering only lattices which satisfy a type of non-orthogonality condition. The results in \cite{ert1} and \cite{ert2} for lattices of the form $\mathbb{Z}\times\mathbb{Z}$ were examined by Gordon and Kappeler in \cite{kg1} and \cite{kg2}. When the lattice satisfies an type of non-orthogonality condition, the analysis is a bit different. We only consider potentials which break down into a finite number of one dimensional finite gap potentials. It was the author's original goal to derive spectral rigidity results when the decomposition into one dimensional potentials contained a one dimensional potential with infinitely many gaps.  The analysis here implies it would be difficult to derive spectral rigidity for such a class of potentials with the current machinery available.  We use the invariants coming from spectral asymptotics of the heat trace in any dimension. We review the one dimensional spectral theory first. The standard references for the one dimensional theory are given by \cite{winkler} and \cite{PT}. For a more modern reference reviewing the notation we refer the reader to Kappeler \cite{birk}. Koroteyv has also proved stronger characterizations of the one dimensional potentials in terms of the gap lengths of the spectra in \cite{k2}, and \cite{k1}, than the ones presented here. It would be interesting if explicitly calculable invariants two dimensional operators which did not involve decomposition to one dimensional operators existed. 

In the sequel to \cite{ert1},  \cite{ert2}, Eskin, et al., show that there is a set of analytic potentials satisfying the conditions (1) and (2) which are dense in $C^{\infty}(\mathbb{R}^2/\mathbb{L})$ such that if $q(x)$ is in this set, then $q(x)$ is Floquet rigid. Furthermore, there is a smaller, but still dense set of analytic potentials in $C^{\infty}(\mathbb{R}^2/\mathbb{L})$ such that if $q(x)$ is in this set and $\tilde{q}(x)$ is Floquet isospectral to $q(x)$ then, $\tilde q(x)=q(\pm x+a)$ where $a$ is an arbitrary constant.  Under the assumptions (1) and (2), if a potential in $\mathbb{R}^2$ is spectrally rigid (resp. unique) then it is Floquet rigid (resp unique), so their results are also true with the words "Floquet rigid" (resp. unique) replacing "isospectrally rigid" (resp unique). The main result of this paper is to show that there is a more general class of potentials which satisfy the conditions for Floquet rigidity than in \cite{ert2}. 

%In the latter portion of this paper we use induction on the results concerning potentials in $\mathbb{R}^n$ to show that there is a class of smooth potentials which are Floquet rigid and an open dense subset of the class of $C^6(\mathbb{R}^n/L)$ potentials. With certain assumptions on the lattice, we show they are also spectrally rigid. We can  conclude Floquet (resp isospectral) uniqueness for a small class of such potentials (with certain conditions on the lattice). 

\section{The Isospectral Manifold in $\mathbb{R}^1$}

In $\mathbb{R}^1$ the structure of the isospectral sets of periodic potentials has been well studied and contains many results which are useful in higher dimensions.   In $\mathbb{R}^1$ the Schr\"odinger operator becomes Hill's operator. 
\begin{align*}
-\frac{d^2}{ds^2}+q(s)
\end{align*}
where $q(s)$ has period $1$ and is real-valued.  We start by assuming that $q$ is at least three times differentiable, so that we can use many of the standard results which may be found in Magnus and Winkler, \cite{winkler}. For the rest of this paper, we will also assume that $q(x)$ has mean zero. We look at the set of $\lambda$ where there is a solution to 
\begin{align}\label{onead}
&-\frac{d^2\phi(s)}{ds^2}+q(s)\phi(s)=\lambda \phi(s) \\
&\nonumber \phi(s+1)=(-1)^m\phi(s). 
\end{align}
The scalars $\lambda$ are known as the periodic and anti-periodic eigenvalues. Through curious use of notation,
the scalar, $\lambda_m^{\pm}$, denotes the eigenvalue corresponding to the eigenfunction $\phi_m^{\pm}(s+1)=(-1)^m\phi_m^{\pm}(s)$ so that 
\begin{equation}\label{spacing}
\lambda_0<\lambda_1^-\leq \lambda_1^+<\lambda_2^-\leq \lambda_2^+ . . . 
\end{equation}
 Hence the periodic spectrum consists of $\{\lambda_m^{\pm},\, m ~\mathrm{even}\}$ and the antiperiodic spectrum is $\{\lambda_m^{\pm},\, m ~\mathrm{odd}\}$. 
 
 If we change the problem (\ref{onead}) so that $\phi(s)$ obeys the boundary condition 
\[ 
\phi(0)=\phi(1)=0,
\]
 then the associated spectrum is called the Dirichlet spectrum. The Dirichlet spectrum are denoted $\mu_{m}(q)$ and they interlace the periodic and anti-periodic spectra. We will often use the fact 
 \begin{align}\label{asymptotic}
 |\lambda_m^+-\lambda_n^+|=\mathcal{O}(|m^2-n^2|),
 \end{align}
 and find it worthwhile to mention it here. Although $\lambda_m^+< \lambda_{m+1}^-$, it is possible to have $\lambda_m^-=\lambda_m^+$.  The spectrum of 
 \begin{align*}
 -\frac{d^2}{ds^2}+q(s)
 \end{align*}
 as an operator in $L^2(\mathbb{R})$ is 
 \begin{align*}
 \bigcup\limits_{m=0}^{\infty}[\lambda_m^+,\lambda_{m+1}^-]
 \end{align*} 
 Each of the intervals $[\lambda_m^+,\lambda_{m+1}^-]$ in the union above is called a "band", or interval of stability. The complement of the set of bands is union of the intervals $(\lambda_m^-, \lambda_m^+)$ which are called "gaps" or intervals of stability. In each gap, the operator 
$ -\frac{d^2}{ds^2}+q(s)$ does not have a bounded eigenfunction. A gap is referred to as open whenever $\lambda_m^-<\lambda_m^+$ and closed if $\lambda_m^-=\lambda_m^+$. The length of a gap is denoted as $\gamma_m$. 
 
 In \cite{garnett} Garnett and Trubowitz  gave a compete characterization of the gaps for $q$ in $L^2_{\mathbb{R}}[0,1]$. 
 \begin{thm}\label{gt} \cite{garnett}
 Let $\gamma_n$, $n\geq 1$, be any sequence of nonnegative numbers satisfying 
 \begin{align*}
 \sum\limits_{n\geq 1}\gamma_n^2<\infty
\end{align*}
 Then there is a way of placing the sequence of open tiles of lengths $\gamma_n$, $n\geq 1$ in order on the positive axis $(0,\infty)$ so that the complement is the set of bands for a function $q$ in $L^2_{\mathbb{R}}[0,1]$. In other words, the map
 \begin{align}\label{bianalytic}
 q\rightarrow \gamma(q)=\{\gamma_n(q)\}_{n\geq 1},
 \end{align}
 from $L^2_{\mathbb{R}}[0,1]$ to $(l^2)^+$, is onto. 
\end{thm}
Furthermore if we multiply the gap lengths $\gamma_m$ by $\epsilon$ where $\epsilon$ is in $[0,1]$ then the map (\ref{bianalytic}) is still onto.  The fundamental result in $\mathbb{R}^1$ is that the set of analytic periodic potentials $M(\epsilon)$ with the same periodic and anti-periodic spectra is equivalent to a torus with dimension equal to $I$ \cite{mckean}. Here $I$ is the number of $m$ for which $\lambda_m^-<\lambda_m^+$. The coordinates $\alpha_{m}(q)$, on this manifold with $m$ referring to the $m^{th}$ gap on $q(s)$, are related to the Dirichlet spectra and the gap lengths. They are defined as follows
\begin{equation}\label{coord}
\sin^2\alpha_{m}(q)=\frac{\mu_{m}(q)-\lambda_{m}^{-}}{\lambda_{m}^+-\lambda_{m}^-} \qquad -\frac{\pi}{2}<\alpha_{m}\leq \frac{\pi}{2}
\end{equation}
where $\mu_{m}(q)$ is the Dirichlet eigenvalue for $q$ such that $\lambda_m^-\leq \mu_m(q)\leq \lambda_m^+$. These coordinates are further discussed in Section 4.  

%and $\alpha_m \in (0,\frac{\pi}{2})$ or $\alpha_m\in (-\frac{\pi}{2},0)$ according to whether or not $q$ is in the upper or lower sheet of $T$. Provided that we can show the coordinates given (\ref{coord}) are still analytic functions, then we can combine the two results to show that the set of potentials $q(\vec{\epsilon},s,\alpha)$ is still  a manifold in terms of the parameter $\epsilon_m$ should we chose to parametrize the gap lengths instead by $\epsilon_m\gamma_m$. We will prove this result in Section 3 for two dimensional potentials but the proof is the same. 

Finally we will need the fact that all the gap lengths are exponentially decreasing if and only if $q(s)$ is real analytic. Whenever $q$ has only a finite number of open gaps, then $q$ must be real analytic, \cite{trubowitz}. The analyticity of $q(s)$ with finitely many gaps is crucial in many of the proofs of the theorems in this paper. 

\section{Review of Necessary Results in $\mathbb{R}^n$}
We outline some necessary results and definitions from \cite{ert1} and \cite{ert2} which will be used in the rest of this paper. 
Let $\mathbb{L}$ be an $n$-dimensional lattice generated by $n$ vectors $v_1, v_2, . . . ,v_n$. We can then consider it's dual $\mathbb{L}^*$ where
\[
\mathbb{L}^*=\{\delta \in \mathbb{R}^n: \delta\cdot v \in \mathbb{Z},  \forall v\in \mathbb{L}\},
\]
to be generated by some basis $\delta_1, \delta_2, . . .,\delta_n$. A function is periodic over the lattice $\mathbb{L}$ if $q(x+d)=q(x)$ for all $d$ in $\mathbb{L}$.  For any arbitrary lattice $\mathbb{L}$ satisfying condition (2) and basis fixed as above, let $\mathbb{S}^*$ be the set of fundamental directions for $\mathbb{L}$, that is
\[
\mathbb{S}^*=\{\delta \in \mathbb{L}^*: \delta\cdot d=1 \text{ for some } d\in \mathbb{L}\}.
\]
It is clear that whenever $\delta$ is in $\mathbb{S}^*$ then $-\delta$ is also in this set, so we reduce the set to $\mathbb{S}$ by only picking $\delta$ in $\mathbb{S}^*$. Therefore any  element of $\mathbb{L}^*/\{0\}$ has a unique representation as $m\delta$ with $\delta$ in $\mathbb{S}$ and $m$ in $\mathbb{Z}$.

If $q$ is a function which is periodic over $\mathbb{L}$, then it has the following Fourier series representation
\[
q(x)=\sum\limits_{\delta\in \mathbb{L}^*} a_{\delta}\exp\sparen{2\pi i\delta \cdot x}
\]
with 
\[
a_{\delta}=\frac{1}{Vol(\Gamma)}\int\limits_{\Gamma}q(x)\exp\sparen{-2\pi i\delta \cdot x}\,dx
\]
where $\Gamma$ the fundamental domain of the lattice $\mathbb{L}$ as given by 
\[
\Gamma=\{ s_1v_1+. . . + s_n v_n: 0\leq s_i\leq 1\}.
\]
If we write 
\begin{align*}
|\delta|^2q_{\delta}(s)=\sum\limits_{n\in \mathbb{Z}}a_{n\delta}\exp(2\pi i n s)
\end{align*}
then we have that 
\[
q(x)=\sum\limits_{\delta\in \mathbb{S}}\sum\limits_{n\in\mathbb{Z}}a_{n\delta}\exp\sparen{2\pi i n\delta \cdot x}=\sum\limits_{\delta \in \mathbb{S}}|\delta|^2q_{\delta}(\delta\cdot x)
\]
where each $q_{\delta}(s)$ is a periodic potential on $\mathbb{R}^1$. These one-dimensional potentials $q_{\delta}(s)'s$ are called directional potentials. 
The assumption that $q(x)$ has mean zero is equivalent to setting $a_{0}=0$ for all the directional potentials. 

%For all $d$ in $\mathbb{L}$, we also let 
%\[
%q_{d}(x)=\int\limits_0^1 q(x+td)\,dt=\sum\limits_{\substack{\delta\in \mathbb{S}\\ \delta\cdot d=0}} q_{\delta}(\delta\cdot x).
%\]
%We then refer to $q_{d}(x)$ as the reduced potential. It follows that repeated reduction in the directions $d_1, d_2, . . . ,d_{n-1}$ generating the lattice results in the following relationship between directional and reduced potentials
%\[
%q_{d_1, d_2, . . . ,d_{n-1}}=q_{\delta_n}(\delta_n\cdot x).
%\]

Theorem 2 in (\cite{ert1}, \cite{ert2}) states that 
\begin{thm}\label{ert}
Spec$(-\Delta+q)$ determines
\begin{align*}
\mathrm{Spec}_k\sparen{-\frac{d^2}{ds^2}+q_{\delta}(s)}\qquad \forall \delta\in \mathbb{S}, k\in \mathbb{R}
\end{align*}
\end{thm}
The theorems in $\mathbb{R}^1$ we mentioned will help reduce the study of periodic potentials in $\mathbb{R}^n$ to the study of $\mathbb{R}^1$ potentials, about which much more is known. 

\section{Potentials in $\mathbb{R}^2$}\label{section2}

Following \cite{ert2}, for the rest of this paper we assume that the elements of the lattice $\mathbb{L}$ satisfy condition (2) as stated in the introduction, and we consider analytic periodic potentials $q(x)$ such that $q(x+d)=q(x)$ for all $d$ in $\mathbb{L}$. We also only consider potentials with a finite number of directional potentials. For this section, we make the additional assumptions that the number of gaps in each direction $\delta_j$ is finite, and that there are at least $3$ directions. This setup differs from \cite{ert2}  where two of the directional potentials were fixed translates of the one gap potentials and the other directions were viewed as perturbations of the zero potential. 

%In $\mathbb{R}^2$ directional and reduced potentials are the same so that $q_{d_j}(x)=q_j(\delta_j\cdot x)$.

Under these assumptions we can simplify the form of $q(x)$ as follows
\begin{align}\label{reduction}
q(x)=\sum_{j=1}^{S}|\delta_j|^2q_j(\delta_{j}\cdot x).
\end{align}
Each one dimensional directional potential $q_j(\delta_j\cdot x)$ corresponds to a one dimensional operator with corresponding eigenvalue  and eigenfunction pair $(\lambda,\phi(s))$ satisfying
\begin{align}\label{eigenequation}
-\frac{d^2}{ds^2}\phi(s)+q_j(s)\phi(s)=\lambda\phi(s).
\end{align}

In order to simplify the computations needed in this paper we make the following assumptions (*)
\begin{enumerate}
\item $\delta_3=\delta_1+ \delta_2$
\item $q_1, q_2$ and $q_3$ have the same number of open gaps 
\end{enumerate}
We will discuss how, given sufficient time and energy, using spectral invariants and the standard perturbation techniques that one could remove the assumptions (*).  The invariants are derived from the trace theorems. If we let the fundamental solution of the heat equation 
\begin{align}
\frac{\partial u}{\partial t}=\Delta u- qu \qquad u(0,x)=f(x)
\end{align}
on $\mathbb{R}^n$ be $G(x,y,t)$ then 
\begin{align}
\sum\limits_{\lambda\in \mathrm{Spec}_k}\exp(-\lambda t)=\sum\limits_{d\in \mathbb{L}}\exp(-2\pi i k\cdot d)\int\limits_{\Gamma}G(x+d,x,t)\,dx
\end{align}
Therefore if one knows Spec$_k(-\Delta+q)$ for all $k$, then one knows 
\begin{align}
\int\limits_{\Gamma}G(x+d,x,t)\,dx \qquad \forall t>0, d\in \mathbb{L}
\end{align}
In \cite{ert1} and \cite{ert2}, they derive Theorem \ref{ert} from the asymptotics of 
\begin{align}\label{heatkernel}
\int\limits_{\Gamma}G(x+Nd+e,x,t)\,dx \qquad \forall t>0, d\in \mathbb{L}
\end{align}
as $N\rightarrow \infty$. 

Theorem \ref{ert}  has the consequence that the set of real-analytic $\tilde{q}(x)$ isospectral to $q(x)$ can be identified with a subset of a real analytic manifold 
\[
M=T_1\times T_2\times . . . \times T_S.
\]
Here each torus $T_j$ has dimension equal to the number of open gaps associated to each directional potential $q(\delta_j\cdot x)$; we call this set $I_j$. This manifold $M$ has dimension $\sum\limits_j|I_j|=N$. Again, the coordinates on the manifold $\alpha_{j,m}(q)$ are given for each $j$ by (\ref{coord}).

 In our case, we would like our set of potentials which we will call $M(\epsilon)$ to have open gap lengths which are parametrized as follows. Let $E_0$ denote the set
\begin{align*}
\{(j,m): (j,m)= (1,1), (2,1)\},
\end{align*}
and $E_1$ denote the set 
\begin{align*}
\{(j,m):  j\leq 2, m>1\}.
\end{align*}
Now we let $\epsilon$ be the vector with four components $(\epsilon_1,\epsilon_2,\epsilon_3,\epsilon_4)$ so we can parametrize the new gap lengths so they depend on $\vec{\epsilon}$ and $\gamma$ as follows 
\begin{align*}
\gamma_{j,m}(\epsilon,\gamma)=\epsilon_j\gamma_{j,m} \,\,  \mathrm{for} \, (j,m)\in E_0 &\\
\gamma_{j,m}(\epsilon,\gamma)=\epsilon_4\gamma_{j,m} \,\,  \mathrm{for} \,\, (j,m)\in E_1 &\\
\gamma_{3,m}(\epsilon,\gamma)=\epsilon_3\gamma_{3,m} \,\,  \mathrm{for} \,\,  m\in I_3&\\
\gamma_{j,m}(\epsilon,\gamma)=\epsilon_4\gamma_{j,m} \,\,  \mathrm{for} \,\, j>3, m\in I_j
\end{align*}
and are associated with the potential $q(\epsilon,x,\alpha)$. Here,  suppressing the "$q$", we have $\alpha=\{\alpha_{j,m}\}$ is the rescaled vector of coordinates, where for each directional potential, the coordinates are given by (\ref{coord}). Notice that we have also written our gap lengths in terms of finitely many parameters and this does not destroy the fact the mapping (\ref{bianalytic}) is onto and in this case analytic.  

The following spectral invariants are derived from higher order terms in the asymptotics of \ref{heatkernel} in \cite{ert2} which we will use in our computations: 
\begin{thm}\label{ert4}
The periodic and anti-periodic spectra for the one dimensional potentials $q_{\delta}(x)$ which form $q(x)$ and the invariants 
\begin{align}\label{invariant}
&\Phi_{\delta_j,m}(\epsilon,\alpha)
\\&\nonumber =\Phi_{j,m}(\epsilon,\alpha)=\int_{\Gamma}|h(\epsilon,x,\alpha)|^2(\phi_{j,m}^{\pm}(\epsilon, \delta_j\cdot x,\alpha))^2\,dx
\end{align}
when  $\lambda_{j,m}^+>\lambda_{j,m}^-$ and 
\begin{align}\label{invariantclosed}
&\Phi_{\delta_j,m}(\epsilon,\alpha)=\Phi_{j,m}(\epsilon,\alpha)
\\&\nonumber = \int_{\Gamma}|h(\epsilon,x,\alpha)|^2\sparen{(\phi_{j,m}^{+}(\epsilon, \delta_j\cdot x, \alpha))^2+(\phi_{j,m}^{-}(\epsilon,\delta_j\cdot x,\alpha))^2}\,dx
\end{align}
when  $\lambda_{j,m}^+=\lambda_{j,m}^-$ maybe recovered from the spectra of $q(x)$. Here $\alpha=\{\alpha_{j,m}\}$ is the collection of coordinates associated to each gap length and we have set
\[
h(\epsilon, x, \alpha)=\sum\limits_{\substack{e \in \mathbb{S} \\ e\cdot d_j\neq 0}}\frac{e}{e\cdot d_j}q_{e}(\epsilon, e\cdot x, \alpha)
\]
with $\delta_j \cdot d_j =0$, and $d_j$ of minimal length. 
\end{thm}

Setting $\Phi_{\delta_j,m}^+(\epsilon,\alpha)=\Phi_{j,m}(\epsilon,\alpha)$, then the number of invariants with $\lambda_m^+>\lambda_m^-$ has dimension equal to the manifold $M(\epsilon)$. We would like to show that the Jacobian determinant of the invariants with respect to the coordinates $\alpha$ is nonzero so that we may apply the implicit function theorem. 

We will primarily be calculating the spectral invariants for potentials at a specific parameter $\epsilon=\epsilon_0$. We let $\epsilon_0$ be the vector with $(\epsilon_1,\epsilon_2,0,0)$ where $\epsilon_1$ and $\epsilon_2$ are in $(0,1)$. When $\epsilon=\epsilon_0$ the potential $q(\epsilon_0,x,\alpha)$ has 
\begin{align*}
\gamma_{j,m}(\epsilon_0,\gamma)=\epsilon_j\gamma_{j,m} \,\,  \mathrm{for} \, (j,m)\in E_0 &\\
\gamma_{j,m}(\epsilon_0,\gamma)=0 \,\,  \mathrm{for} \, (j,m)\in E_0^c
\end{align*}
for gap lengths. The potential $q(\epsilon_0,x,\alpha)$  is therefore the sum of $2$ potentials with only one gap, one in each direction $\delta_j$, $j=1,2$. The rest of the directional potentials are zero.  While the limit $q(\epsilon_0,x,\alpha)$ coincides with the form of the potential as calculated in \cite{ert2}, one specific difference remains- the first two directional have finitely many gaps, they are not just translates of the $\wp$ function. 
We will Taylor expand the Jacobian determinant with respect to $\epsilon_3$ around $\epsilon\neq \epsilon_0$ and use these computations to show that the Jacobian determinant for certain fixed $\alpha$ is not identically zero. 

For the rest of this paper,  we let  $\wp(s+\frac{i\tau}{2},\tau)$ denote a general normalized Weierstrass $\wp$ function. Whenever the parameter $\tau$ is real and greater than zero, then $\wp(s+\frac{i\tau}{2},\tau)$ is real-valued with periods $1$ and $\tau$ \cite{complex}. The real-valued $\wp$-function is always even about $\frac{1}{2}$, and by a theorem of Hochstadt \cite{hochstadt}, all one gap potentials are translates of the $\wp$-function. The directional potential, in the limit, $q_j(\epsilon_0,s,\alpha)=\wp(s+\frac{i\tau_j}{2}+\nu_j,\tau_j)$ has eigenfunctions which satisfy the following equation: 
\begin{equation*}
-\frac{d^2}{ds^2}\phi(\epsilon_0,s,\alpha)+q_j(s)\phi(\epsilon_0,s,\alpha)=\lambda\phi(\epsilon_0,s,\alpha).
\end{equation*}
where $q_j(\epsilon_0,0,\alpha)=\wp(\frac{i\tau_j}{2}+\nu_j,\tau_j)$ has bands given by
\begin{equation}\label{spectrum}
[-\wp\sparen{\frac{1}{2}},-\wp\sparen{\frac{i\tau_j+1}{2}}]\cup[-\wp\sparen{\frac{i\tau_j}{2}},+\infty).
\end{equation}
Aligning the classical elliptic function theory with spectral theory \cite{kaihua} we have that, 
\begin{align}
 -\wp\sparen{\frac{1}{2}}=\lambda_0 \qquad -\wp\sparen{\frac{i\tau_j+1}{2}}=\lambda_1^- \qquad -\wp\sparen{\frac{i\tau_j}{2}}=\lambda_1^+.
\end{align}
We will need the parameters $\tau_j$ later in the computation of the Fourier coefficients of the $\wp$ function and the perturbation calculations for the eigenfunctions. From equation (\ref{spectrum}) we know that they are related to the $\epsilon_j$ as follows
\begin{align}\label{gapone}
\wp\sparen{\frac{i\tau_j+1}{2}}-\wp\sparen{\frac{i\tau_j}{2}}=\epsilon_{j}\gamma_{j,1}
\end{align}
for $j=1,2$. Therefore if we pick $\epsilon_j$, we pick $\tau_j$ and vice versa. 

Since any potential $q(x,\epsilon,\alpha)$ is always Floquet isospectral to $q(\pm x+ a,\epsilon,\alpha)$ where $a$ is arbitrary, we cannot hope to remove the sign or translation degeneracy. We know that when $\epsilon=\epsilon_0$ that  $\delta_1\cdot a=\nu_1$ and $\delta_2\cdot a=\nu_2$, so for simplicity we fix $a$ so when $\epsilon=\epsilon_0$ then $a=0$. As a result we have that 
\[
q_j(s,\alpha,\epsilon_0)=\wp_j(s+\frac{i\tau_j}{2},\tau_j)=\sum\limits_{n\in \mathbb{N}}a_n^j\cos(2\pi n s) 
\] 
for $j=1,2$, where the coefficients $a_n^j$ are given by Appendix \ref{fcop}. We consider our manifold $M(\epsilon)$ of potentials which have translation fixed as above. 

In order to prove that $M(\epsilon)$ actually is an analytic manifold with coordinates $\alpha=\{\alpha_{j,m}(q)\}$ we must first remind the reader of a few definitions involved in the selection of the coordinates $\{\alpha_{j,m}\}$ defined by (\ref{coord}) as they are related to the Dirichlet spectra $\mu_{j,m}(q)$ of the operator. 
We define the discriminant $\Delta(\lambda)$ as follows
\begin{equation}\label{discriminant}
\Delta^2(\lambda)-4=4(\lambda_0-\lambda)\prod\limits_{n=1}^{\infty}\frac{(\lambda_{n}^+-\lambda)(\lambda_{n}^--\lambda)}{n^4\pi^4}.
\end{equation}
Let $\mu_{m}(s, q_j)=\mu_{j,m}(\epsilon,s,\alpha)$ be the the solution to the system (where here we are suppressing the $j$)
\begin{align}\label{system}
&\frac{d\mu_{m}(\epsilon,s,\alpha)}{ds}
= m^2\pi^2\frac{\sqrt{\Delta^2(\mu_m)-4}}{\prod\limits_{\substack{n\in I,\\ n\neq m}}(\mu_{n}(\epsilon,s,\alpha)-\mu_{m}(\epsilon,s,\alpha))/n^2\pi^2}
\end{align}
with $\mu_{j,m}(\epsilon,0,\alpha)=\mu_{m}(0,q_j), k\in I$, where the choice of signs is initially by the sign of numerator, and changes whenever $\mu_{j,m}(\epsilon,s,\alpha)$ hits $\lambda_{j,m}^{\pm}$. The proof of analyticity of $\mu$ by examining (\ref{system}) remains almost exactly the same as in \cite{ert2} and is omitted here. Since there are a finite number of coordinates, it is easy to see that analyticity in each coordinate is preserved, and hence $M(\epsilon)$ is still an analytic manifold

By McKean-Van Moerbeke \cite{mckean}, the initial value the sum of the initial values, $\mu_{j,m}(\epsilon,0,\alpha)$, is related to each directional potential $q_j(\epsilon,0,\alpha)$ in the following way 
\begin{equation*}
q_j(\epsilon,0,\alpha)=\lambda_{0}+\sum\limits_{m\in I_j}(\lambda_{j,m}^++\lambda_{j,m}^--2\mu_{j,m}(\epsilon,0,\alpha))
\end{equation*}
and this relationship remains true when the parameter $s$ is varied
\begin{equation}\label{flowq}
q_j(\epsilon,s,\alpha)=\lambda_{0}+\sum\limits_{m\in I_j}(\lambda_{j,m}^++\lambda_{j,m}^--2\mu_{j,m}(\epsilon,s,\alpha)).
\end{equation} 
Using a combination of formulas on pp. 325 and 329, in \cite{trubowitz}, the eigenfunctions for each directional potential corresponding to $\lambda_{j,m}^+$ for all $j$ can be written as
\begin{equation}\label{eigenfunction}
(\phi_{m}^+(\epsilon,s,\alpha))^2=\prod\limits_{n\in I_j}\sparen{\frac{\lambda_{m}^+-\mu_{n}(\epsilon,s,\alpha)}{\lambda_{m}^+-\dot{\lambda}_{n}}}
\end{equation}
where $\dot{\lambda}_{m}$ is the zero of $\frac{\partial\Delta}{\partial\lambda}$ lying between $\lambda_{m}^-$ and $\lambda_{m}^+$. It is important to note here that the formula in \cite{ert2} is a misprint. 
%
% We also have that the eigenfunction corresponding to $\lambda_m^-$ is 
%\begin{equation}\label{eigenfunctionminus}
%(\phi_{m}^-(\epsilon,s,\alpha))^2=\prod\limits_{n\in I_j}\sparen{\frac{\lambda_{m}^--\mu_{n}(\epsilon,s,\alpha)}{\lambda_{m}^--\dot{\lambda}_{n}}}
%\end{equation}
We will also need the derivatives of the eigenfunctions which from equation (\ref{eigenfunction}) are 
\begin{equation}\label{eigenfunctionderivative}
2\phi_{m}^+(\epsilon,s,\alpha)\frac{d\phi_{m}^+(\epsilon,s,\alpha)}{ds}=\sum\limits_{n\in I_j}\frac{-1}{\lambda_n^+-\dot{\lambda}_k}\sparen{\frac{d\mu_n(\epsilon,s,\alpha)}{ds}}\prod\limits_{k\neq n}\frac{\lambda_m^+-\mu_k(\epsilon,s,\alpha)}{\lambda_m^+-\dot{\lambda}_k}
\end{equation}
with the derivative for $\phi^-(\epsilon_0,s,\alpha)$ computed similarly.
Let us start by considering the eigenfunctions for those directional potentials with $j>3$. Because we are looking for the root between $\lambda_{j,m}^+$ and $\lambda_{j,m}^-$ when $\epsilon= \epsilon_0$, we make the substitution $\lambda=\lambda_{j,m}^-+\epsilon_4\gamma_{j,m}\tilde{\lambda}$ into (\ref{discriminant}) to find that 
\[
\Delta^2(\tilde{\lambda})-4=\epsilon_4^2\tilde{\lambda}(1-\tilde{\lambda})f(\epsilon_4\tilde{\lambda},\epsilon_4)
\]
where $f(z,\epsilon_4)$ is analytic and $f(0,0)=\gamma_{m}^2\neq 0.$ Therefore for $\epsilon_4$ sufficiently small, $\dot{\lambda}_{m}$ corresponds to the root of 
\[
0=(1-2\tilde{\lambda})f(\epsilon_4\tilde{\lambda},\epsilon_4)+\epsilon_4\tilde{\lambda}(1-\tilde{\lambda})\frac{\partial f}{\partial z}(\epsilon_4\tilde{\lambda},\epsilon_4)
\] 
near $\tilde{\lambda}=\frac{1}{2}$. As a result, the following estimate holds 
\begin{equation}\label{ratio}
\frac{\lambda_{m}^+(\epsilon)-\lambda_m^-(\epsilon)}{\lambda^+_{m}(\epsilon)-\dot{\lambda}_{m}(\epsilon)}=2+\mathcal{O}(\epsilon_4).
\end{equation}
giving that 
\begin{equation}\label{main}
\frac{\lambda_{m}^+(\epsilon_0)-\mu_{m}(\epsilon_0,\alpha,s)}{\lambda^+_{m}(\epsilon_0)-\dot{\lambda}_{m}(\epsilon_0)}=2\cos^2(\tilde{\alpha}_{m}(s,\alpha)).
\end{equation}
The variable $\tilde{\alpha}_m(s,\alpha)$ denotes the solution to the system (\ref{system}) where  $\epsilon=\epsilon_0$ with initial condition $\alpha$ under the change of variables (\ref{coord}).  The same estimates above are true for the eigenfunctions $\phi_{3,m}^+(\epsilon_0,s,\alpha)$, $m$ in $I_3$ when expanded with respect to $\epsilon_3$.  We can conclude for all $j\geq 3$
\begin{equation}\label{eigenplus}
\phi_{j,m}^+(\epsilon_0,s,\alpha)=\sqrt{2}\cos\tilde{\alpha}_{j,m}(s,\alpha)
\end{equation}
where we know we have picked the right sign by verifying the derivative (\ref{eigenfunctionderivative}) in the limit. 

Now we consider the case when $j\leq 2$. When $\epsilon=\epsilon_0$, we have for all $n>1$ that  $\lambda_{j,n}^+=\lambda_{n}^-=\mu_{n}=\dot{\lambda}_{n}$ so that  terms in the product (\ref{eigenfunction}) where $n\neq m$ and $n>1$ become 
\begin{equation}\label{error}
\frac{\lambda_{j,m}^+(\epsilon_0)-\mu_{j,n}(\epsilon_0,s,\alpha)}{\lambda_{j,m}^+(\epsilon_0)-\dot{\lambda}_{j,n}(\epsilon_0)}=1,
\end{equation}
and for $n=1$ we have
\begin{equation}
\frac{\lambda_{j,m}^+(\epsilon_0)-\mu_{j,1}(\epsilon_0)}{\lambda_{j,m}^+(\epsilon_0)-\dot{\lambda}_{j,1}(\epsilon_0)}=\frac{\lambda_{j,m}^+(\epsilon_j)-\lambda_{j,1}^-(\epsilon_j)-\epsilon_j\gamma_{j,1}\sin^2(\tilde{\alpha}_{j,1}(s,\alpha))}{\lambda_{j,m}^+-\dot{\lambda}_{j,1}(\epsilon_j)}.
\end{equation}
Combining equations (\ref{main}) (which is still true for $j\leq 2$) and (\ref{error}), we see that for $\epsilon=\epsilon_0$, and $(j,m)$ in $E_1$,
\begin{equation}\label{eigenm}
(\phi_{j,m}^+(\epsilon_0,\alpha,s))^2=2\cos^2(\tilde{\alpha}_{m}(s,\alpha))\sparen{\frac{\lambda_{j,m}^+(\epsilon_j)-\lambda_{j,1}^-(\epsilon_j)-\epsilon_j\gamma_{j,1}\sin^2(\tilde{\alpha}_{j,1}(s,\alpha))}{\lambda_{j,m}^+(\epsilon_j)-\dot{\lambda}_{j,1}(\epsilon_j)}}.
\end{equation}
Comparing with the derivative computed in (\ref{eigenfunctionderivative}) we know that the correct choice of sign is 
\begin{equation}\label{signchoice}
(\phi_{j,m}^+(\epsilon_0,\alpha,s))=\sqrt{2}\cos(\tilde{\alpha}_{m}(s,\alpha))\sqrt{\frac{\lambda_{j,m}^+(\epsilon_j)-\lambda_{j,1}^-(\epsilon_j)-\epsilon_j\gamma_{j,1}\sin^2(\tilde{\alpha}_{j,1}(s,\alpha))}{\lambda_{j,m}^+(\epsilon_j)-\dot{\lambda}_{j,1}(\epsilon_j)}}.
\end{equation}

%and similarly (since $\lambda_m^+=\lambda_m^-$ in the limit)
%\begin{equation}\label{eigenmminus}
%(\phi_{j,m}^-(\epsilon_0,\alpha,s))^2=2\sin^2(\tilde{\alpha}_{j,m}(s,\alpha))\sparen{\frac{\lambda_{j,m}^+-\lambda_{j,1}^--\epsilon_j\gamma_{j,1}\sin^2(\tilde{\alpha}_{j,1}(s,\alpha))}{\lambda_{j,m}^+-\dot{\lambda}_{j,1}(\epsilon_j)}}
%\end{equation}
%However for $(j,m)$ in $E_0$ we have
%\begin{equation}\label{eigenone}
%(\phi_{j,1}^+(\epsilon_0,\alpha,s))^2=2\cos^2\tilde{\alpha}_{j,1}(s,\alpha).
%\end{equation}
%and also 
%\begin{equation}\label{eigenoneminus}
%(\phi_{j,1}^-(\epsilon_0,\alpha,s))^2=2\sin^2\tilde{\alpha}_{j,1}(s,\alpha).
%\end{equation}

The introduction of this setup provides the necessary background to introduce the following theorem:
\begin{thm}\label{rigidtwo}
For all but an analytic set of $(\epsilon_3,\epsilon_4)$ in $[0,1]^2$, there is an open set of potentials satisfying the hypotheses (1),(2) and (*) in $M(\epsilon)$ which are isospectral to only a finite number of other analytic potentials.  
\end{thm}
In order to find the Jacobian corresponding to the invariants as given by equation (\ref{invariant}), we must first figure out what it means to calculate their derivatives with respect to $\{\alpha_{j,m}\}$ with $(j,m)$ in $E_0^c$. We start with the following lemma
\begin{lem}\label{lemmaone}
%For the diagonal entries $(r,k)=(j,m)$ in $E_1$ we have
%\begin{align}\label{derivealpha1}
%\frac{\partial\tilde{\alpha}_{j,m}(\delta_j\cdot x,\alpha)}{\partial\alpha_{j,m}}=\sparen{\frac{d\tilde{\alpha}_{j,m}(\delta_j\cdot x,\alpha)}{ds}}\sparen{\frac{d\tilde{\alpha}_{j,m}(\delta_j\cdot x,\alpha)}{ds}|_{0}}^{-1}.
%\end{align}
%and for the diagonal entries with $(j,m)$ not in $E_1$ we have 

For  $(j,m)$ in $E_0^c$, we have
\[
\frac{\partial\tilde{\alpha}_{j,m}(s,\alpha)}{\partial\alpha_{j,m}}=1, \qquad \mathrm{and} \qquad 
\frac{\partial\tilde{\alpha}_{j,m}(s,\alpha)}{\partial\alpha_{r,k}}=0 \qquad \mathrm{when}\,\, (r,k)\neq (j,k)
\]
%unless $(j,m)$ is in $E_0$ and $(r,k)=(j,1)$ in which case 
%\begin{equation}\label{evolve}
%\frac{\partial\tilde{\alpha}_{j,m}(\delta_j\cdot x,\alpha)}{\partial\alpha_{j,1}}=\frac{\lambda_m^+-\lambda_1^--\epsilon_1\gamma_1\sin^2\tilde{\alpha}_{j,1}(s,\alpha)}{\lambda_m^+-\lambda_1^--\epsilon_1\gamma_1\sin^2\alpha_{j,1}}.
%\end{equation}
\end{lem}

\begin{proof}
 Examining (\ref{system}) under the change of variables given by (\ref{coord}) for $(j,m)$ in $E_1$ and $\epsilon=\epsilon_0$

\begin{equation}\label{flowalpham}
\frac{d\tilde{\alpha}_{j,m}(s,\alpha)}{ds}=\frac{\sqrt{(\lambda_{j,m}^+-\lambda_{0})(\lambda_{j,m}^+-\lambda_{j,1}^+)(\lambda_{j,m}^+-\lambda_{j,1}^-)}}{\lambda_{j,m}^+-\lambda_{j,1}^--\epsilon_j\gamma_{j,1}\sin^2\tilde{\alpha}_{j,1}(s,\alpha)}
\end{equation}
Therefore $\tilde{\alpha}_{j,m}(s,\alpha)$ depends only on $\alpha_{j,1}$ and the initial data for $\tilde{\alpha}_{j,m}(0,\alpha)=\alpha_{j,m}$
so the result follows. 

The case whenever $j\geq 3$ and $\epsilon=\epsilon_0$, is much easier to compute. We have for all such corresponding $m$
\begin{equation}\label{flowalphazero}
\frac{d\tilde{\alpha}_{j,m}(s)}{ds}=m\pi
\end{equation}
so again the result follows by the same reasoning above. 
\end{proof}
For the computations done in the appendix, we need to know that when $\epsilon_j=0$, (\ref{signchoice}) agrees with the limit one would expect. In other words for $(j,m)$ in $E_1$, we have
\begin{align}
\phi_{j,m}^+(\epsilon_0,\alpha,s)=\sqrt{2}\cos(\pi m s+\alpha_{j,m})+\mathcal{O}(\epsilon_j) 
\end{align}
which is easily verifiable by Lemma \ref{lemmaone}, and the estimates (\ref{ratio}) and (\ref{error}). We have computed the eigenfunctions in (\ref{signchoice}) to illustrate that they are expressed in terms of elliptic functions, and therefore the invariants will not be explicitly computable. 

We can now prove the main Lemma. If we consider a potential $q(\epsilon,x,\alpha)$ in $M(\epsilon)$ then it is associated to a fixed set of coordinates $\alpha$. Let $\det(J)(\epsilon,\alpha)$ be the Jacobian determinant of the invariants $\Phi_{j,m}(\epsilon,\alpha)$ with respect to the coordinates $\{\alpha_{j,m}\}$ with ${j,m}$ in $E_0^c$, and $\det(J)(\epsilon,\alpha)$ is an $(N-2)\times(N-2)$ determinant. 

The proof of Theorem \ref{rigidtwo} will be based on the following Lemma:
\begin{lem}\label{mainlemma}
There is a choice of $\epsilon_1,\epsilon_2$ in $[0,1]$ such that on a dense open set of $\alpha$,
\begin{align}
\det(J)(\epsilon,\alpha)\neq 0
\end{align} 
\end{lem}
\begin{proof}
We will proceed by showing that for all $k=1$ to $n-1$ 
\begin{align*}
\frac{\partial^k \det(J)}{\partial\epsilon_3^k}(\epsilon_0,\alpha)=0
\end{align*}
while
\begin{align*}
\frac{\partial^n \det(J)}{\partial\epsilon_3^n}(\epsilon_0,\alpha)\neq 0
\end{align*}
where $n=|I_1|+|I_2|-2=|E_1|$. The desired result will follows since we notice that if for some $n$
\begin{align*}
\frac{\partial^n \det(J)}{\partial\epsilon_3^n}(\epsilon_0,\alpha)\neq 0
\qquad \mathrm{and}\qquad
\det(J)(\epsilon,\alpha)\equiv 0
\end{align*}
then this is a contradiction since all of the derivatives of $\det(J)(\epsilon,\alpha)$ evaluated at any $\epsilon$ should be identically zero as well, since $\det(J)(\epsilon,\alpha)$ is an analytic function of $\epsilon$. 

Now we proceed to calculate the derivatives of $\det(J)(\epsilon,\alpha)$. Let the columns $v_i(\epsilon,\alpha)$ of $\det(J)(\epsilon,\alpha)$ be indexed by $i$ where $i$ ranges from $1$ to $N-2$. Each $i$ corresponds to a pair of indices $(j,m)$ such that 
\[
v_i(\epsilon,\alpha)=\nabla_{\alpha}\Phi_{j,m}(\epsilon,\alpha)
\]
where we are considering the pairs $(j,m)$ ordered first by the $j$ and then by the $m$. The perturbation calculations to find the derivatives of the invariants are located in Appendices. In order to examine the Jacobian further, we need the following key observations:
\begin{enumerate}
\item 
$\frac{\partial q_j}{\partial\alpha_{l,k}}(\epsilon_0,\delta_j\cdot x,\alpha)=0 \qquad \forall (l,k)\in E_o^c, \,\,\mathrm{and} \,\, \forall j$
\item 
$\frac{\partial (\phi_{j,m}^+)^2}{\partial\alpha_{l,k}}(\epsilon_0,\delta_j\cdot x,\alpha)=0 \qquad \forall (l,k), (j,m) \in E_0^c\,\, \mathrm{unless}\, (j,k)=(l,m)$
\item
$\frac{\partial q_j}{\partial\epsilon_3}(\epsilon_0,\delta_j\cdot x,\alpha)= \frac{\partial (\phi_{j,m}^+)^2}{\partial\epsilon_3}(\epsilon_0,\delta_j\cdot x,\alpha)=0 \qquad  \forall j\neq 3$
\end{enumerate}
The first two observations follow from Lemma \ref{lemmaone} and formulae (\ref{flowq}) and (\ref{eigenfunction}), respectively. The last observation follows from the parametrization of the open gaps since only $q_3(\epsilon,\delta_3\cdot x,\alpha)$ and $\phi_{3,m}(\epsilon,\delta_3\cdot x,\alpha)$ for $m$ in $I_3$ depend on $\epsilon_3$.  

Going back to equation (\ref{invariant}), each invariant has the form as follows
\begin{align}\label{invarform}
\Phi_{j,m}(\epsilon,\alpha)=\int\limits_{\Gamma}\sabs{\sum\limits_{\substack{l\in N\\ l\neq j}}\frac{\delta_l}{\delta_l\cdot d_j} q_l(\epsilon,\delta_l\cdot x,\alpha)}^2
(\phi^+_{j,m}(\epsilon, \delta_j\cdot x,\alpha))^2\,dx
\end{align}
Now we let $D$ denote a generic constant independent of the coordinates.  When $\epsilon=\epsilon_0$ the form of the invariants (\ref{invarform}) for $j\geq 3$ coincides with that of \cite{ert2}. Since $\delta_1$ and $\delta_2$ form a basis for $\mathbb{S}$, we know that there exists a nonzero pair of integers $(p_l,r_l)$ such that for any third vector $\delta_l \neq \delta_1, \delta_2$ we have $\delta_l= p_l\delta_1+r_l\delta_2$. Therefore when $j\geq 3$
\begin{align}
&\Phi_{j,m}(\epsilon_0,\alpha)= D\int\limits_0^1 \int\limits_ 0^1 (\wp_2(t+\frac{i\tau_2}{2}, \tau_2)) (\wp_1(s+\frac{i\tau_1}{2}, \tau_1))\cos^2(\pi m (p_js+r_j t)+\alpha_{j,m})  \,ds\,dt+D
\end{align}
Exactly as in \cite{ert2}, we have that when $(j,m)$ is such that $j\geq 3$ 
\[
\Phi_{j,m}(\epsilon_0,\alpha)=c_{1,2,j}a_{mp_j}^1a_{mr_j}^2\cos 2\alpha_{j,m}+D
\]
 The coefficients $c_{1,2,j}a_{mp_j}^1a_{mr_j}^2$ are independent of the coordinates and nonzero. They can be found in \ref{fcop}. However for $j$ in $\{1,2\}$, we come across the degeneracy that 
\begin{align}\label{itszero}
\frac{\partial\Phi_{j,m}}{\partial\alpha_{l,k}}(\epsilon_0,\alpha)=0
\end{align}
for all $(l,k)$ in $E_0^c$.
We know from our observations (1) and (2) that (\ref{itszero}) holds  except for possibly when $(l,k)=(j,m)$. In this case since again $\delta_1$ and $\delta_2$ form a basis for $\mathbb{S}$ we can write 
\begin{align}\label{zero}
&\frac{\partial\Phi_{j,m}}{\partial\alpha_{j,m}}(\epsilon_0,\alpha)=\int\limits_{\Gamma} \sabs{\frac{\delta_l}{\delta_l\cdot d_j}q_l(\epsilon_0,\delta_l\cdot x,\alpha)}^2\frac{\partial(\phi^+_{j,m})^2}{\partial\alpha_{j,m}}(\epsilon_0,\delta_j\cdot x,\alpha))\,dx\\=& \nonumber D\int\limits_0^1 \wp^2_l(s+\frac{i\tau_l}{2},\tau_l)\,ds \int\limits_ 0^1\frac{\partial(\phi_{j,m}^+)^2}{\partial\alpha_{j,m}}(\epsilon_0,t,\alpha)\,dt
\end{align}
where $l\neq j$ and $l$ is in $\{1,2\}$. But since we consider our eigenfunctions as normalized for all $(j,m)$, e.g. $||\phi_{j,m}^+(\epsilon,\delta_j\cdot x,\alpha)||_{L^2(\mathbb{R})}=1$,  the right hand side of (\ref{zero}) is just zero. 
 
Therefore for all $i$ from $1$ to $n$ we have
\begin{align*}
v_i(\epsilon_0,\alpha)=0.
\end{align*}
while for all $i$ from $n+1$ to $(N-2)$ we see that
\begin{align}\label{lower}
\sparen{v_i(\epsilon_0,\alpha)}_l^t=\left\{ \begin{array}{rcl} & 0   & l=1, . . , i-1 \\ & c_{1,2,j}a_{mp_j}^1a_{mr_j}^2\sin 2\alpha_{j,m}  & l=i
\\ & 0  & l>i
\end{array}\right \}.
\end{align}
Because the determinant is a multi-linear function of its rows,  we may write 
\begin{align*}
\det (J)(\epsilon_0,\alpha)=\det\sparen{v_1, v_2, . . .,v_n,v_{n+1}, . . ,v_{N-2}}
\end{align*}
It is now clear that for all $k=1$ to $n-1$ 
\begin{align*}
\frac{\partial^k \det(J)}{\partial\epsilon_3^k}(\epsilon_0,\alpha)=0
\end{align*}
however for $k=n$ we have 
\begin{align}\label{det}
\frac{\partial^n \det(J)}{\partial\epsilon_3^n}(\epsilon_0,\alpha)=C(n)\det\sparen{\frac{\partial v_1}{\partial \epsilon_3},\frac{\partial v_2}{\partial \epsilon_3}, . . ,\frac{\partial v_n}{\partial \epsilon_3},v_{n+1}, . . . ,v_{N-2}}.
\end{align}
where $C(n)$ is a constant depending on $n$ only. 

From observations (1-3) we know for $j$ in $\{1,2\}$ 
\begin{align*}
\frac{\partial^2\Phi_{j,m}}{\partial\epsilon_3\partial\alpha_{l,k}}(\epsilon_0,\alpha)=0
\end{align*}
except for possibly when $l=3$ or $(l,k)=(j,m)$. We then note that corresponding rows with $1\leq i\leq n$ in (\ref{upper}) take the form 
\begin{align}\label{upper}
\sparen{\frac{\partial v_i}{\partial \epsilon_3}}_l^t =\left\{ \begin{array}{rcl} & 0  & l=1, . . , i-1 \\ & \frac{\partial^2\Phi_{j,m}}{\partial\alpha_{j,m}\partial\epsilon_3}(\epsilon_0,\alpha)  & l=i  \\ & 0  & r>l>i
\\ & \frac{\partial^2\Phi_{j,m}}{\partial\alpha_{3,j}\partial\epsilon_3}(\epsilon_0,\alpha) & i= r. . . k
\\ & 0  & l>r
\end{array}\right \}
\end{align}
Here the index $r$ corresponds to $(3,1)$ and $k-r=|I_3|$. We can conclude from (\ref{lower}) and (\ref{upper}) the determinant (\ref{det}) is an upper triangular one. The determinant (\ref{det}) looks like
\[
\left| {\begin{array}{cc}
 A & B \\
 0 & C
 \end{array} } \right|
\]
where $A$ is an $n\times n$ block diagonal matrix, and $C$ is an $(N-n-2)\times (N-n-2)$ block diagonal matrix. If the diagonal entries in the upper triangular determinant (\ref{upper}) are nonzero, then we will arrive at the desired result that 
\begin{align}
\frac{\partial^n \det(J)}{\partial^n\epsilon_3}(\epsilon_0,\alpha)\neq 0
\end{align}
The collection of diagonal entries for $(j,m)$ in $E_1$ corresponding the block $A$, for $1\leq i\leq n$ are  
$\frac{\partial^2\Phi_{j,m}}{\partial\epsilon_3\partial\alpha_{j,m}}(\epsilon_0,\alpha).$
From \ref{perturb}, we know that there is a choice of $\epsilon_1$ and $\epsilon_2$ so that these invariants are nonzero except on an analytic set of $\alpha_{j,m}.$
Also from  \ref{perturb} and equation (\ref{invar3}), whenever $i>n$ we have diagonal entries corresponding to $(j,m)$ with $j\geq 3$, corresponding to the block $C$ are 
\begin{align}\label{jgeq3}
\frac{\partial\Phi_{j,m}}{\partial\alpha_{j,m}}(\epsilon_0,\alpha)=-2c_{1,2,j}a_{mp_j}^1a_{mr_j}^2\sin 2\alpha_{j,m} 
\end{align}
These entries are only zero whenever $\alpha_{j,m}\equiv 0\mod \pi/2$ for $j\geq 3$. The lemma is finished. 

$\mathbf{Remark:}$ It should be possible to remove the assumption (*) by using the standard perturbation series to calculate $(\phi_{j,m}^+(\epsilon_j,s,\alpha))^2$ around $\epsilon_j=0$.  If $\delta_3$ were generically of the from $p_3\delta_1+r_3\delta_2$,  then we conjecture that (\ref{awful1}) is nonzero provided we expanded the eigenfunctions to order $n$ with $n$ satisfying the relation $m\pm l=np_3$ or $m\pm l=nr_3$ for some $l$ in $\mathbb{N}$. The calculations required to do so are difficult. This conjecture is discussed further in \ref{perturb}
\end{proof}

\begin{proof}[Proof of Theorem \ref{rigidtwo}] 
This proof is very similar to the one in \cite{ert2} and is again included for completeness.  Let us start by assuming the matrix $J$ is invertible on $M(\epsilon)$ except for on an analytic set, say $U$, of $(\epsilon_3,\epsilon_4)$ Recall that on the manifold $\epsilon_{j}$ and the corresponding $\alpha_{j,1}$ for $j=1,2$ are fixed. Then given some $\tilde{\epsilon}$ with variable components $(\epsilon_3,\epsilon_4)$ in $[0,1]^2/U$, we let 
\[
F=\{\alpha: \frac{\partial\Phi}{\partial\alpha}(\tilde{\epsilon},\alpha)=0\}.
\]
Since
\[
\Phi(\tilde{\epsilon},\alpha): M(\tilde{\epsilon})\rightarrow \mathbb{R}^{N-2},
\]
the corollary follows if we can show that the set  $\Phi^{-1}(\Phi(F)^c)$ is open and dense. We know the set is open since $\Phi^{-1}$ is open, and $F$ is compact. If we assume that it is not dense, then the set contains contains an open set $O$ which also contains a point $\alpha_0$ which is not in $F$.  Because the Jacobian is nonzero, $\Phi$ is a homeomorphism on a neighborhood of $\alpha_0$, which implies $\Phi(F)$ contains an open set. The last statement contradicts Sard's theorem. Now we assume that $\Phi(\alpha_1)$ is not in $\Phi(F)$ and $\Phi^{-1}(\Phi(F))$ is infinite. Let $\alpha_2$ be an accumulation point of $\Phi^{-1}(\Phi(\alpha_1))$. Because $\Phi$ is continuous, $\Phi(\alpha_2)=\Phi(\alpha_1)$ and $\frac{\partial\Phi}{\partial\alpha_2}\neq 0.$ It follows that there is a neighborhood, $N$, of $\alpha_2$ such that $\alpha$ is in $N$ and $\Phi(\alpha)=\Phi(\alpha_2)$ implies $\alpha=\alpha_2$. This is a contradiction to our assumption so we know $\Phi^{-1}(\Phi(\alpha_1))$ is finite. Because $\Phi$ is a spectral invariant, then $\Phi^{-1}(\Phi(F)^c)$ is a subset of the manifold which satisfies the conditions of Theorem \ref{rigidtwo}. 
\end{proof}

This theorem has a nice corollary if we make the following observations:
\begin{enumerate}
\item  Any two directions $\delta_1$ and $\delta_2$ form a basis for the lattice $\mathbb{L}$, so our choice of basis and translate is arbitrary. 
\item The potentials on $M(\epsilon)$ satisfying the conditions of the theorem are dense in the set of all analytic potentials in the $C^{\infty}$ topology.
\item The set of smooth periodic potentials which are a sum of only a finite number of directional potentials each with a finite number of gaps in each direction are dense in the set of finite gap periodic potentials in the $C^{\infty}(\mathbb{R}^2/\mathbb{L})$ topology. 
\item The set of finite gap potentials is dense in the set of all $C^6(\mathbb{R}^2/\mathbb{L})$ potentials in the $C^{\infty}$ topology.
\end{enumerate}

\begin{cor} The set of analytically rigid potentials is dense in the set of smooth potentials on $\mathbb{R}^2/\mathbb{L}$ in the $C^{\infty}(\mathbb{R}^2/\mathbb{L})$ topology
\end{cor}            

\appendix

\section{Fourier Coefficients of the $\wp$ Function}\label{fcop}

As detailed in section 1, the $\wp$-function depends on a parameter $\tau_j>0$. The complex valued function $\wp(z,\tau)$ is given by
\[
\wp(z,\tau)=\frac{1}{z^2}+\sum\limits_{(m,n)\in \mathbb{Z}^2/0}\sparen{\frac{1}{(z-n-im\tau)^2}-\frac{1}{(n+im\tau)^2}}
\]
which as before is real on the line $x+\frac{i\tau_j}{2}$ and setting,
\[
a=e^{-2\pi \tau_j}\quad b=e^{2\pi i(x+\frac{i\tau_j}{2})}
\]
gives
\[
\frac{1}{(2\pi i)^2}\wp(x,\tau)=\frac{1}{12}+\sum\limits_{n=-\infty}^{\infty}\frac{ab}{(1-a^mb)^2}-2\sum\limits_{n=1}^{\infty}\frac{na^n}{1-a^n}.
\]
Because
\[
\frac{a^mb}{(1-a^mb)^2}=\sum_{n=1}^{\infty}n(a^mb)^n \quad m \geq 0
\]
and 
\[
\frac{a^mb}{(1-a^mb)^2}=\sum_{n=1}^{\infty}n(a^{-m}b^{-1})^n \quad m< 0
\]
the representation 
\begin{align*}
&\frac{1}{(2\pi i)^2}\wp(x,\tau)
\\= &\frac{1}{12}+\sum\limits_{n=1}^{\infty}na^{\frac{n}{2}}e^{2\pi inx}+\sum\limits_{m=1}^{\infty}\sum\limits_{n=1}^{\infty}n(a^{n(m+\frac{1}{2})}e^{2\pi inx}+a^{n(m-\frac{1}{2})}e^{-2\pi ix})-2\sum\limits_{n=1}^{\infty}\frac{na^n}{1-a^n}.
\end{align*}
Changing the order of summation we get
\[
\frac{-1}{4\pi^2}\wp(x,\tau)=\frac{1}{12}+\sum\limits_{n=1}^{\infty}\frac{2na^{\frac{n}{2}}}{1-a^n}\cos(2\pi nx)-2\sum\limits_{n=1}^{\infty}\frac{na^n}{1-a^n}.
\]
Therefore the Fourier coefficients for the $\wp$ functions in the first three directions  are given by:
\begin{align}\label{fseriesforp}
&a_n^j=\frac{-8\pi^2n\exp(-\pi n\tau_j)}{1-\exp(-2\pi n\tau_j)} \qquad \mathrm{for}\,\, n\geq 1 \\
& a_0=-\frac{\pi^2}{3}+8\pi^2\sum\limits_{n=1}^{\infty}\frac{n\exp(-\pi n\tau_j)}{1-\exp(-\pi n\tau_j)}
\end{align}
where $j=1,2$. The appropriate $\tau_j$ will depend on the choice of $\epsilon_j$ as given in section 1.

\section{Calculation of the Invariants}\label{perturb}

In order to prove Lemma \ref{mainlemma} we need to show that there exist $\epsilon_1$ and $\epsilon_2$ in $[0,1]$ such that
\begin{align}
\frac{\partial^2\Phi_{1,m}}{\partial\epsilon_3\partial\alpha_{1,m}} (\epsilon_0,\alpha) \qquad \mathrm{and}\qquad \frac{\partial^2\Phi_{2,n}}{\partial\epsilon_3\partial\alpha_{2,n}} (\epsilon_0,\alpha)
\end{align}
are nonzero except perhaps on an analytic set of $\alpha$. 

We know by (\ref{invariant})
\begin{align}
\Phi_{j,m}(\epsilon,\alpha)=\int\limits_{\Gamma}\sabs{\sum\limits_{j\neq k}\frac{\delta_k}{\delta_k\cdot d_j}q_k(\epsilon,\delta_k\cdot x,\alpha)}^2(\phi_{j,m}^+(\epsilon,\delta_j\cdot x,\alpha))^2\,dx
\end{align}
Each $q_j(\epsilon,\delta_j\cdot x,\alpha)$ is independent of $\epsilon_3$ when $j\neq 3$.  Furthermore since $q_k(\epsilon,\delta_k\cdot x,\alpha)$ and $(\phi_{k,m}^+(\epsilon,\delta_j\cdot x,\alpha))^2$ are independent of $\mu_{j,m}(\epsilon,\delta_j\cdot x,\alpha)$ for all $j\neq k$, so only the function $(\phi_{j,m}^+(\epsilon,\delta_j\cdot x,\alpha))^2$ depends on $\alpha_{j,m}$ in the above integral. As a result we can write
\begin{align}\label{beforezero}
\frac{\partial^2\Phi_{j,m}}{\partial\epsilon_3\partial\alpha_{j,m}}(\epsilon,\alpha)=\int\limits_{\Gamma}\frac{\partial}{\partial\epsilon_3}\sabs{\sum\limits_{j\neq k}\frac{\delta_k}{\delta_k\cdot d_j}q_k(\epsilon,\delta_k\cdot x,\alpha)}^2\frac{\partial}{\partial\alpha_{j,m}}(\phi_{j,m}^+(\epsilon,\delta_j\cdot x,\alpha))^2\,dx
\end{align}

Whenever $\epsilon=\epsilon_0$, then $q_3(\epsilon_0,\delta_3\cdot x,\alpha)=0$ and the derivative $\partial_{\epsilon_3}q_3(\epsilon_0,\delta_3\cdot x,\alpha)$ can be calculated using the Fredholm alternative as in \cite{ert2}.  Following Appendix I of \cite{ert2}, we may write
\begin{align}\label{deriveq3}
\frac{\partial q_3}{\partial\epsilon_3}(\epsilon_0,\delta_3\cdot x,\alpha)=\sum\limits_{n\in I_3}\gamma_{3,n}\cos(2\pi \delta_3\cdot x+2\alpha_{3,n}).
\end{align}
Also from the derivation of equation (\ref{eigenplus}), we can conclude that 
\begin{align}\label{eigenorder}
(\phi_{j,m}^+(\epsilon_0,s,\alpha))^2=2\cos^2(\pi m(\delta_j\cdot x)+\alpha_{j,m})+\mathcal{O}(\epsilon_j) 
\end{align} 
where by Lemma \ref{lemmaone} the order terms are bounded by $\epsilon_jC$ where $C$ depends only on $\alpha_{j,m}$.
 Hence from analytic perturbation theory and the derivation of (\ref{eigenplus}) we can use (\ref{eigenorder}) to conclude that 
 \begin{align}
 \frac{\partial(\phi_{j,m}^+)^2}{\partial\alpha_{j,m}}(\epsilon_0,\delta_j\cdot x,\alpha)=-2\sin(2\pi (\delta_j\cdot x)m+2\alpha_{j,m})+\mathcal{O}(\epsilon_j)
 \end{align}
 where the  $\mathcal{O}(\epsilon_j)$ terms are bounded by $\epsilon_jC$ with $C$ a constant depending only on the coordinate $\alpha_{j,m}$.
Because any two directions $\delta_1$ and $\delta_2$ in $\mathbb{S}$ form a basis, we know that there exists a nonzero pair of integers $(p_l,r_l)$ such that for any third vector $\delta_l \neq \delta_1, \delta_2$ we have $\delta_l= p_l\delta_1+r_l\delta_2$.  For easier computations we make  the initial variable change $\delta_1\cdot x=s$ and $\delta_2\cdot x=t$, with the associated Jacobian, Vol$(\Gamma)$, and rewrite the invariants. 
We also let $D$ denote a generic constant which is independent of the coordinates, and we let
\begin{align}
c_{l,k,j}=\frac{\delta_l\cdot\delta_j}{2(\delta_l\cdot d_j)(\delta_k\cdot d_j)}(\mathrm{Vol}(\Gamma)).
\end{align} 

From statements (1-3) in Section 3, (\ref{deriveq3}), (\ref{eigenorder}) and (\ref{beforezero}), when $\epsilon=\epsilon_0$, we have
\begin{align}\label{awful1}
 &(c_{3,l,j}\mathrm{Vol}(\Gamma))^{-1}\frac{\partial^2\Phi_{j,m}}{\partial\epsilon_3\partial\alpha_{j,m}}(\epsilon_0,\alpha)=\\&
 \\& 4\nonumber\int\limits_0^1\int\limits_0^1\sparen{\sum\limits_{n\in I_3}\gamma_{3,n}\cos(2\pi n(s+t)+2\alpha_{3,n})}\wp_l(t+i\frac{\tau_l}{2},\tau_l))\frac{\partial(\phi_{j,m}^+)^2}{\partial\alpha_{j,m}}(\epsilon_0,s,\alpha))^2\,ds \,dt=\\& \nonumber
2 \sum\limits_{n\in I_3}\gamma_{3,n}a_n^l\int\limits_0^1\cos(2\pi ns+2\alpha_{3,n})\frac{\partial(\phi_{j,m}^+)^2}{\partial\alpha_{j,m}}(\epsilon_0,s,\alpha)\,ds
\end{align}
where $0\leq j,l\leq 2, j\neq l$.

When $j=1$, by the hypothesis (*) on the number of open gaps that $q_3$ has, the right hand side of (\ref{awful1}) is just
\begin{align}
2a_m^2\gamma_{3,m}\sin(2\alpha_{3,m}-2\alpha_{1,m})+\mathcal{O}(\epsilon_1)
\end{align}
Here the $\mathcal{O}(\epsilon_1)$ terms are bounded by $\epsilon_1C$ where the constant depends only on $\alpha_{1,m}$ and $\alpha_{3,n}$ for all $n\in I_3$.  We recall that $a_n^l\rightarrow 0$ as $\epsilon_l\rightarrow 0$ for all $n$ in $\mathbb{N}$ and $l=1,2$ since $a_n^l$ is related to $\epsilon_l$ by Equation (\ref{gapone}) and (\ref{fseriesforp}) However, we can make the constant  uniform in  $\epsilon_2$.  If we let
\begin{align}\label{Mm}
\sup_{s\in[0,1]} \sabs{\frac{\partial(\phi_{1,m}^+)^2}{\partial\alpha_{1,m}}(\epsilon_0,s,\alpha)}=M_m<\infty
\end{align}
then this follows from the rough estimate
\begin{align}
&|\sum\limits_{n\in I_3}\int\limits_0^1\int\limits_0^1\sparen{\sum\limits_{n\in I_3}\gamma_{3,n}\cos(2\pi n(s+t)+2\alpha_{3,n})}p_2(t+i\frac{\tau_2}{2},\tau_2))\\& \times \nonumber \sparen{\frac{\partial(\phi_{1,m}^+)^2}{\partial\alpha_{1,m}}(\epsilon_0,s,\alpha))^2-\sin(2\pi ms+2\alpha_{1,m})} \,ds \,dt| \leq \\& \nonumber \sum\limits_{n\in I_3}\gamma_{3,n}a_n^2\cos(2\alpha_{3,n})\sparen{M_m+2}\leq  2n\sparen{M_m+2}
\end{align}
since the gap lengths $\gamma_{3,n}$ and the Fourier coefficients $a_n^2$ are exponentially decreasing.  Now let $\beta$ in $(0,1)$ be a small fixed parameter. We consider the set of $\alpha$ such that 
\begin{align}
|2\alpha_{3,m}-2\alpha_{1,m}-k\pi|\geq \beta  \qquad \forall k\in \mathbb{Z}, \, m\in I_1
\end{align}
We let this set be denoted as $A_1$, and note that its complement is an analytic set. 
Therefore provided we chose $\epsilon_1$ and $\epsilon_2$ which satisfy the inequality 
\begin{align}\label{choiceone}
(M_m+2)\epsilon_1<\frac{|a_m^2|\gamma_{3,m}}{2n}\sin(\beta)
\end{align}
for all $m$ in $I_1$ and $\alpha$ in $A_1$ then (\ref{awful1}) is nonzero for $j=1$ and all $m$ in $I_1$. The tricky step is to prove that we can pick $\epsilon_1,\epsilon_2$ in $(0,1)$ such that \ref{choiceone} holds for all $m$ in $I_1$ but also so 
\begin{align}
\frac{\partial^2\Phi_{2,n}}{\partial\epsilon_3\partial\alpha_{2,n}}(\epsilon_0,\alpha)\neq 0
\end{align}
for all $n$ in $I_2$ except on an analytic set of $\alpha$.

Because for small $\epsilon_1$, $a^1_{n_1}>a^1_{n_2}$ whenever $n_2>n_1$ the right hand side of (\ref{awful1}) is already written in ascending order in $\epsilon_1$ for $j=2,l=1$. Let
\begin{align}
b_{j,m,n}(\epsilon_0,\alpha)=\int\limits_0^1\cos(2\pi ns+2\alpha_{3,n})\frac{\partial(\phi_{j,m}^+)^2}{\partial\alpha_{j,m}}(\epsilon_0,s,\alpha)\,ds.
\end{align}
Since we do not know if  $b_{2,m,n}(\epsilon_0,\alpha)\equiv 0$ in $\alpha$ for all $m\neq n$, we pick $\epsilon_1$ as follows. Say $b_{2,m,1}(\epsilon_0,\alpha)$ is nonzero except on an analytic set of $\alpha$, and then let the set where $b_{2,m,1}(\epsilon_0,\alpha)=0$ be denoted as $A_{2,m,1}^c$. If we can prove that for $j=2,l=1$, (\ref{awful1}) is nonzero for some $\alpha$, then it will be nonzero on some open dense set of $\alpha$'s. The easiest $\alpha$ to select is the one when $b_{2,m,1}(\epsilon_0,\alpha)$ is at its maximum. Hence we then pick $\epsilon_1$ such that 
\begin{align}\label{maxone}
\max_{\alpha\in A_{2,m,1}}|\gamma_{3,1}a_1^1 b_{2,m,1}(\epsilon_0,\alpha)|\geq \sabs{ \sum\limits_{\substack{k\in I_3\\k\neq 1}}\gamma_{3,k}a_k^1b_{2,m,k}(\epsilon_0,\alpha)}
\end{align}
where the $\max$ is taken over the possible values of $b_{2,m,1}(\epsilon_0,\alpha)$ with $\alpha$ in $A_{2,m,1}$, and we consider the right hand side of (\ref{maxone}) to be evaluated at this $\alpha$ as well. 
If $b_{2,m,1}(\epsilon_0,\alpha)\equiv 0$ in $\alpha$, but $b_{2,m,2}(\epsilon_0,\alpha)$  is nonzero except on an analytic set of $\alpha_{2,m}$, and let the set where $b_{2,m,2}(\epsilon_0,\alpha)=0$ be denoted as $A_{2,m,2}^c$
 then pick $\epsilon_1$ such that 
\begin{align}
\max_{\alpha\in A_{2,m,2}}|\gamma_{3,2}a_2^1b_{2,m,2}(\epsilon_0,\alpha)|\geq \sabs{ \sum\limits_{\substack{k\in I_3\\k>2}}\gamma_{3,k}a_k^1b_{2,m,k}(\epsilon_0,\alpha)}
\end{align}
where again the $\max$ is taken over the possible values of $b_{2,m,2}(\epsilon_0,\alpha)$ with $\alpha$ in $A_{2,m,2}$. 
We continue this process inductively.  As before, let $\beta$ be a small parameter in $(0,1)$. We now also consider the set of $\alpha$ such that 
\begin{align}
|2\alpha_{3,m}-2\alpha_{2,m}-k\pi|\geq \beta  \qquad \forall k\in \mathbb{Z}, \, m\in I_2
\end{align}
and let this set be denoted by $A_{2,m,m}$. 
We know 
\begin{align}
b_{2,m,m}(\epsilon_0,\alpha)=\sin(2\alpha_{3,m}-2\alpha_{2,m})+\mathcal{O}(\epsilon_2)
\end{align}
where the $\mathcal{O}(\epsilon_2)$ terms are bounded by $\epsilon_2C$ where  is a constant depending only on $\alpha_{2,m}$ and $\alpha_{3,n}$ for all $n$ in $I_3$.  Hence our selection process terminates because $b_{2,m,m}(\epsilon_0,\alpha)$ is not zero for $\alpha$ in $A_{2,m,m}$ provided we chose $\epsilon_2$ such that 
\begin{align}\label{choicetwo}
\epsilon_2|C|<\sin(\beta)
\end{align}
Hence we pick $\epsilon_1$ in terms of $\epsilon_2$ so that 
\begin{align}\label{choice}
\min_{n}\max_{\alpha\in A_{2,m,n}}\sparen{|\gamma_{3,n}a_n^1b_{2,m,n}(\epsilon_0,\alpha)|}\geq \sabs{\sum\limits_{\substack{l\in I_3\\k> n}}\gamma_{3,k}a_k^1b_{2,m,k}(\epsilon_0,\alpha)}
\end{align}
for all $m$ in $I_2$ where the $\min_{n}$ is taken over those indices $n$ for which $b_{2,m,n}(\epsilon_0,\alpha)$ is not identically zero in $\alpha$. This choice of $\epsilon_1$ and $\epsilon_2$ is not in contradiction to our choice of $\epsilon_1$ small compared to $\epsilon_2$ since the right hand side of the inequality (\ref{choice}) always has a higher order function of $\epsilon_1$ than the left hand side. Furthermore $b_{2,m,n}=0$ for all $m\neq n$ whenever $\epsilon_2=0$, so the right hand side is bounded. We conjecture using a computer and the standard perturbation series for $b_{j,m,n}(\epsilon_0,\alpha)$ that the assumption $q_1$, $q_2$ and $q_3$ have the same number of gaps could be removed. However, this is computationally difficult since it has been verified $b_{j,m,n}(\epsilon_0,\alpha)$ is $\mathcal{O}(\epsilon_j^{|m-n|})$ for all $m$ up to some sufficiently large values of $m$ and $n$.  

For the case with $j\geq 3$, the invariants are computed almost exactly the same way as in \cite{ert2} because the form of the invariants coincides for these indices. In this case we have that 
\begin{align}\label{invar3}
\Phi_{j,m}(\epsilon_0,\alpha)=c_{1,2,j}a^1_{mp_j}a^2_{mr_j}\cos(2\alpha_{j,m})+D
\end{align}

\section{References}
\bibliography{Aldenbibiso}
\end{document}